\documentclass[12pt,thmsa, reqno]{amsart}

\usepackage{graphicx}
\usepackage{amssymb, euscript}
\usepackage[square, comma, sort&compress, numbers]{natbib}
\usepackage{amsmath}
\usepackage{multicol}
\usepackage{lineno}

\usepackage{setspace}
\def\sn{\bbs^{n-1}}

\def\Cal{\mathcal}

\def\T{{\Cal T}}

\def\bbr{{\Bbb R}}

\def\bbd{{\Bbb D}}

\def\bbs{{\Bbb S}}

\def\const{{\hbox{\rm const}}}

\def\rn{\bbr^n}

\def\part{\partial}
\def\intl{\int\limits}

\def\Gam{\Gamma}
\def\Om{\Omega}
\def\a{\alpha}
\def\om{\omega}

\def\Del{\Delta}

\def\vp{\varphi}

\def\g{\gamma}
\def\gam{\gamma}

\def\sig{\sigma}
\def\lam{\lambda}
\def\z{\zeta}
\def\e{\varepsilon}

\def\t{\tau}
\def\th{\theta}

\font\frak=eufm10

\def\fr#1{\hbox{\frak #1}}

\def\frS{\fr{S}}

\newtheorem{theorem}{Theorem}[section]
\newtheorem{lemma}[theorem]{Lemma}

\theoremstyle{definition}

\theoremstyle{remark}
\newtheorem{remark}[theorem]{Remark}

\theoremstyle{corollary}

\newtheorem{proposition}[theorem]{Proposition}

\numberwithin{equation}{section}

\newcommand{\be}{\begin{equation}}
\newcommand{\ee}{\end{equation}}

\newcommand{\bea}{\begin{eqnarray}}
\newcommand{\eea}{\end{eqnarray}}
\newcommand{\Bea}{\begin{eqnarray*}}
\newcommand{\Eea}{\end{eqnarray*}}

\def\sideremark#1{\ifvmode\leavevmode\fi\vadjust{\vbox to0pt{\vss% the remark
 \hbox to 0pt{\hskip\hsize\hskip1em%                          will appear only
\vbox{\hsize2cm\tiny\raggedright\pretolerance10000%          on the side
 \noindent #1\hfill}\hss}\vbox to8pt{\vfil}\vss}}}%
\begin{document}

\title [Spherical Slice Transform ] {On the Spherical Slice Transform }


\author{Boris Rubin}
 \address{Department of Mathematics, Louisiana State
University,  Baton Rouge, LA, 70803, USA}
\email{borisr@lsu.edu}


%\date{November 26, 2001 and, in revised form 29.03.02.}

%\dedicatory{}


\begin{abstract} 
We study the spherical slice transform $\frS$ which assigns to a function $f$ on the unit sphere $S^n$  in $\bbr^{n+1}$  the  
integrals of $f$ over cross-sections of $S^n$ by $k$-dimensional affine planes passing through the north pole $(0, \ldots, 0,1)$. 
These transforms are known when  $k=n$. We consider all $2\le k\le n$ and obtain an explicit formula connecting $\frS$ with the classical $(k-1)$-plane Radon-John 
transform $R_{k-1}$ on $\rn$.   Using this connection, known facts for $R_{k-1}$, like inversion formulas, support theorems, representation on zonal functions, and some others, are  reformulated for $\frS$.
\end{abstract}
\maketitle

\section{Introduction}

Let  $S^n$ be the unit sphere in $\bbr^{n+1}$, $n\ge 2$. Given a  point $a \in \bbr^{n+1}$ and an integer $k$, $2\le k\le n$, let $\T_a$ be the manifold of 
all $k$-dimensional affine planes $\t$ in $\bbr^{n+1}$ passing through $a$ and intersecting  $S^n$. The {\it spherical slice transform} 
assigns to a function $f$ on $S^n$  the collection  of integrals of $f$ over cross-sections $\t \cap \,S^n$. Specifically, we set
\be\label{Sli}
 (\frS_a f)(\t \cap \,S^n)=\intl_{\t \cap \,S^n} f(\eta) \,dS_\t(\eta), \qquad \t \in \T_a,\ee
where integration is performed with respect to the standard surface area measure on the subsphere $\t \cap S^n$.

Depending on the configuration of cutting planes, one can distinguish different kinds of spherical slice transforms. The classical case when 
 $a=o$, the origin of $\bbr^{n+1}$, has more than one hundred years' history. Such operators are known as the Funk transforms (or Minkowski-Funk transforms, or spherical Radon transforms) 
and play an important role in different branches of  analysis,  geometry, and tomography; see, e.g., \cite {Fu11, Fu13, Ga, GGG1, GRS, H11, Min, P1, P2,  Ru15}.

The study of the operators  $\frS_a$ with $a\neq o$, arising in spherical tomography and spherical integral geometry  is quite recent; see, e.g., \cite{GRS}.
  If $a$  lies strictly inside or outside of $S^n$, the operators (\ref{Sli}) were studied in a series of papers by  
  Agranovsky and Rubin \cite {Ag, AR1, AR2, Ru19},  Quellmalz \cite {Q, Q1},  Salman \cite {Sa16, Sa17}.   

It turns out that the cases $a \in S^n$ and $a \notin S^n$ are essentially different. Specifically, in the second case,   
 $\frS_a$ is non-injective (i.e., has a non-trivial kernel) on standard function spaces, like $C(S^n)$ or $L^p(S^n)$, while in the first case the injectivity is available.  
 The reason for this striking phenomenon is that if $a \notin S^n$,  then  $\frS_a$ is equivalent to the classical Funk transform which annihilates odd functions;  
  see \cite {Ag, AR1, AR2, Ru19} for details.  On the contrary, if  $a \in S^n$, then 
 $\frS_a$ inherits properties of the Radon-John transform over affine planes (see (\ref {Slik}) below), and therefore it is injective.

  In the present paper, we focus on the case $a \in S^n$. Without loss of generality, we assume that $a$ is the north pole $N=(0, \ldots, 0,1)$ 
and write $\frS$ in place of $\frS_a$.

The transformation $\frS$  for $n=2$ was introduced  
by   Abouelaz and  Daher \cite {AbD} who obtained  inversion formulas for $\frS f$, provided that $f$ is continuous,  vanishes identically in a neighborhood 
of  $N$,  and  zonal  (i.e. invariant under rotations about the last coordinate axis).
In this case, $\frS f$ is represented by a one-dimensional Abel-type integral that can be explicitly inverted using the tools of 
 fractional differentiation. A general (not necessarily zonal) case  for $n=2$ was studied by Helgason   \cite [p. 145]{H11} who proved the support theorem and injectivity of  $\frS$
  on smooth  functions  vanishing at $N$ with all derivatives. 
 These results were extended by the author 
 \cite [Section 7.2]{Ru15} to all $k=n\ge 2$  under minimal assumptions for functions in Lebesgue spaces.

 The  term {\it spherical slice transform} was introduced  in  \cite [p. 145]{H11} specifically for $\frS f$. However, 
  it can be equally 
 adopted   to all  Radon-like transforms over plane sections of the sphere. These include  operators  $\frS_a$ with any  $a\in \bbr^{n+1}$, 
 the limiting case  $a=\infty$ when all cross-sections  are parallel  \cite{AR2, HQ, Ru19a}, 
  the spherical section transforms over geodesic subspheres of $S^n$ \cite {Ru00a}, and many others. 
  The list of references at the end of the article may be helpful to the reader who is willing to try his/her hand at some of such  intriguing transforms.

In the present paper, our main concern  is the afore-mentioned operator $\frS $ for arbitrary $2\!\le \!k\!\le \! n$. 
Using the stereographic projection, it will be shown that $\frS$ 
 inherits injectivity and many other properties of 
 the classical Radon-John  transform
\be\label{Slik} (R_{k-1} g)(\z)=\intl_{\z} g(x)\, d_\z x, \qquad \z\in G(n, k-1),\ee
where $G(n, k-1)$ is the manifold of all $(k-1)$-dimensional affine planes in $\rn$ and  $d_\z x $ stands for the standard Euclidean measure on $\z$. 
Our aim is  to make the connection between $\frS$  and $R_{k-1}$ precise and transfer known results for $R_{k-1}g$ to  $\frS f$.
 
 Most of the results of the paper were obtained several years ago,   
 but  remained unpublished. Later developments in  \cite {Ag, AR1, AR2}, related to $\frS_a f$  with  $a \notin S^n$, inspired me to make my  notes available to a wide audience.

 Section 2 contains Preliminaries. In Section 3 we prove the main theorem (see Theorem \ref{th1}) establishing a 
 connection between  $\frS f$ and  $R_{k-1} g$. Section 4 contains corollaries of Theorem \ref{th1}. 
 It includes sharp existence conditions for $\frS f$ in integral terms (Proposition 4.1), representation on zonal functions by Abel-type integrals (Proposition 4.3), 
 support theorem (Theorem 4.4), and inversion formulas (Subsection 4.4). 
 
 For convenience of the reader, we collected some auxiliary facts about Radon-John transforms in the Appendix. 
Here the Support Theorem \ref{hobfu} is new. It 
 contains  a slight, but important, improvement of the remarkable discovery by Kurusa \cite [Theorem 3.1] {Ku94},  establishing essential 
difference between the cases $k=n-1$ and $k<n-1$. 
It turns out that in the case $k<n-1$,  the well-known  rapid decrease condition, as, e.g., in  
 \cite [p. 33]{H11}, can be eliminated. Some results from \cite {Ku94} were generalized in \cite {Ku21}.

\section{Preliminaries}
\subsection{Notation}

In the following, $e_1, \ldots, e_{n+1}$ are the coordinate unit vectors,
\[N=(0, \ldots, 0,1)=e_{n+1}, \qquad \rn= \bbr e_1 \oplus \cdots \oplus \,\bbr e_{n}, \qquad S^{n-1} =S^n \cap \,\rn, \]
$\T$ is the manifold of all $k$-dimensional affine planes $\t$ in $\bbr^{n+1}$ passing through $N$ and  intersecting $\rn$. Abusing notation, we write $|\t|$ for the Euclidean distance from the origin to $\t$;  
  $dS (\cdot)$ (sometimes with a subscript) stands for the  surface area measure on the surface under consideration.  A function $f$ on $S^n$ is called zonal if it is invariant under rotations about the last coordinate axis.

\subsection{Stereographic Projection}
Consider the bijective mapping
\be\label{MUIT} \rn \ni x \xrightarrow{\;\nu\;} \eta\in S^{n}\setminus \{N\},  \qquad \nu (x)=\frac{2x+(|x|^2-1)\,e_{n+1}}
{|x|^2+1}.\ee
 The inverse mapping $\nu^{-1}: S^{n}\setminus \{N\} \rightarrow \rn$ is the  stereographic projection  from the north pole $N$ onto $\rn$.

\vskip 0.1 truecm
\begin{figure}[h]
\centering
\includegraphics[scale=.6]{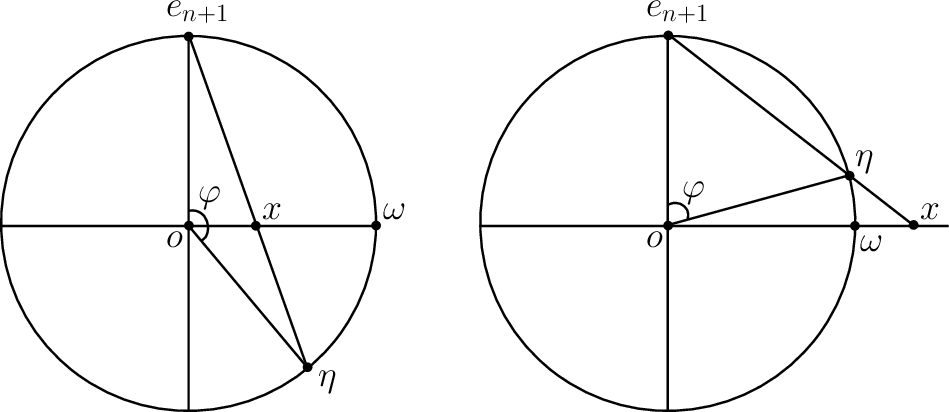}
\caption{$\eta= \om \,\sin \vp + e_{n+1} \cos \, \vp, \quad \om \in S^{n-1}, \quad |x|=\cot (\vp/2)$.} \label{stereo2}
\end{figure}
\vskip 0.1 truecm

The following lemma  is well known; see, e.g.,  \cite [Lemma 7.6]{Ru15}.  We present it here for the sake of completeness and further references.

\begin{lemma} \label {stereoviat} ${}$\hfill

\noindent {\rm (i)} If $f\in L^1 (S^n)$, then
\be \label {stslice}\intl_{S^n} f(\eta)\, d\eta= 2^{n}\intl_{\bbr^{n}} (f\circ \nu) (x)\,\frac{dx}{(|x|^2 +1)^{n}}.\ee

\noindent {\rm (ii)}  If $g\in L^1 (\bbr^{n})$, then
\be \label {stslice1} \intl_{\bbr^{n}} g(x)\, dx=\intl_{S^n} (g\circ\nu^{-1}) (\eta)\,\frac{ d\eta}{(1-\eta_{n+1})^{n}}.\ee
\end{lemma}
\begin{proof} {\rm (i)} Passing to spherical coordinates, we have
\bea
l.h.s.
&=&\intl_0^\pi \sin^{n-1} \vp\, d\vp
\intl_{S^{n-1}} \!\!f(\om\, \sin
\vp + e_{n+1}\, \cos \, \vp)\, d\om\nonumber\\
&{}& \qquad \mbox {\rm ($s=\cot (\vp/2)$)}\nonumber\\
&=& 2^n\intl_0^\infty \frac{s^{n-1} ds}{(s^2 +1)^{n}}\intl_{\sn} \!\!\! f \left (\frac{2s\om  \!+\! (s^2\! -\!1)\,e_{n+1}}{s^2 +1}\right )\, d\om\!=\!r.h.s.\nonumber\eea

 {\rm (ii)} We set $g(x)=2^n (|x|^2 +1)^{-n} (f\circ \nu) (x)$ in (\ref{stslice}). Since
  \be\label{ll3567g} |x|^2 +1= s^2 +1=\cot^2 (\vp/2) +1=\frac{2}{1-\cos \vp}=\frac{2}{1-\eta_{n+1}},\ee
  the result  follows.
\end{proof}

\section{Connection with the Radon-John Transform}

\begin{theorem} \label{th1} Let $2\le k\le n$, $\t \in \T$. Then
\be\label{Slik1}
(\frS f)(\t \cap S^n)=(R_{k-1} g)(\t \cap \rn), \qquad g(x)=\frac{2^{k-1} (f\circ \nu)(x)}{(|x|^2+1)^{k-1}},\ee
provided that either side of the first equality  in (\ref{Slik1}) exists  in the Lebesgue sense.
\end{theorem}

\begin{proof} STEP I. Every $k$-plane $\t \in \T$ is uniquely represented as
\[ \t\equiv \t (\t_0, u)= \t_0 +u, \qquad \t_0 \in G_{n+1, k}, \quad u \in \t_0^\perp,\]
where  $G_{n+1, k}$ is the Grassmann manifold of $k$-dimensional linear subspaces of $\bbr^{n+1}$ and $\t_0^\perp$ is the orthogonal complement of $\t_0$. Let  $\psi$ be the angle between $u$ and $e_{n+1}$. Then 
 \be \label {sVaS}
u=(\th \sin \psi +e_{n+1}\, \cos \psi) \cos \psi, \quad \th \in \sn \subset \rn, \quad  |u|=\cos \psi. \ee
 We also set
\[ r=\sin \psi, \qquad t=\cot \psi.\]
If $\z=\t \cap \rn$ and $G_{n, k-1}$ is  the Grassmann manifold of $(k-1)$-dimensional linear subspaces of $\rn$, then
\be \label {set} \z= \z_0 + v, \quad  \z_0 \in G_{n, k-1}, \quad v \in \z_0^\perp \cap \rn,\ee
where $\z_0^\perp$ is the orthogonal complement of $\z_0$ in $\bbr^{n+1}$. By (\ref{sVaS}),
\be \label {set1} |v|= \frac {|u|}{\sin \psi}=\frac{|u|}{\sqrt{1-|u|^2}}=\cot \psi=t. \ee
Let
\[\tilde \bbr^{k-1}= \bbr e_{n-k+1} \oplus \cdots \oplus \bbr e_{n-1} \]
be the coordinate $(k-1)$-subspace and choose a rotation
\[
\a= \left[\begin{array} {cc} \a_0 & 0\\
0 & 1
 \end{array} \right], \qquad \a_0 \in SO(n),\]
satisfying $ \a_0 (\tilde \bbr^{k-1})=\z_0$, $ \; \a_0 (e_{n})=\th $;
cf. (\ref{sVaS}). Then $\t= \a\t'$ for some $\t' \in \T$ satisfying
\be\label {latz1} \t'=  \t'_0 + u', \quad \t'_0 \in G_{n+1, k}, \quad u' \in (\t'_0)^\perp, \quad \t_0=\a\t'_0, \ee
 $u=\a u'$.  Similarly,
for $\z=\t \cap \rn$ we have
\[  \z=\a\z', \qquad \z'= \t' \cap \rn= \tilde \bbr^{k-1} + te_n.\]
Clearly,
\be\label {lasz} u'= (e_n \sin \psi +e_{n+1}\, \cos \psi) \cos \psi;\ee
cf. (\ref {sVaS}). Let
\[ \gam= \t \cap S^n,  \qquad \gam'= \t' \cap S^n\]
be the corresponding spherical cross-sections. In this notation, $\frS f$ reads
\be\label {lasz1}
(\frS f)(\a\gam')=\intl_{\gam'} (f\circ \a)(\eta') \, dS_{\gam'} (\eta').\ee

STEP II. Let $P_{\t'_0}\gam' $ be the orthogonal projection of $\gam'$ onto the subspace $\t'_0 \in G_{n+1, k}$, as in (\ref{latz1}). We  rotate $\t'_0$ about the subspace $\tilde \bbr^{k-1}= \bbr e_{n-k+1} \oplus \cdots \oplus \bbr e_{n-1}$, making $\t'_0$ coincident with the $k$-dimensional coordinate plane  $\tilde \bbr^{k}= \tilde \bbr^{k-1} \oplus  \bbr e_{n+1}$. This gives
\[
P_{\t'_0}\gam' \equiv \gam' - u' =\rho \tilde \gam,\]
where  $u'$ has the form (\ref{lasz}),  $\tilde \gam$ is the $(k-1)$-sphere of radius $r=\sin \psi$ in the ``vertical'' $k$-subspace $\tilde \bbr^{k}$,  and the rotation $\rho$ has the form
\[
 \rho=\left[ \begin{array} {cc} I_{n-1} &  0  \\ 0 & \rho_\psi
\end{array}\right], \qquad \rho_\psi=\left[ \begin{array} {cc} \sin \psi &  -\cos \psi  \\ \cos \psi & \sin \psi
\end{array}\right]. \]
Setting
\[S^{k-1}=S^n \cap \tilde \bbr^{k}\]
(cf. Figure 2 for $k=n$), from (\ref{lasz1}) we obtain
\[
(\frS f)(\a\gam')=(\frS f)(\a[u'+\rho \tilde \gam])=r^{k-1}\intl_{S^{k-1}} (f\circ \a)(u'+r\rho \sig ) \, dS(\sig) .\]

\vskip 0.1 truecm
\begin{figure}[h]
\centering
\includegraphics[scale=.7]{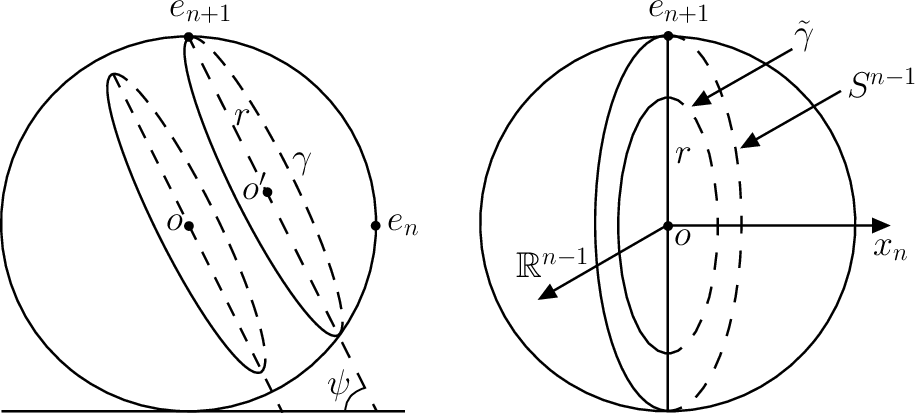}
\caption{$\gam =o'+\rho \tilde \gam, \; o'=u', \; r=\sin \psi$.} \label{stereo2}
\end{figure}
\vskip 0.1 truecm

Denote $f_\a (\cdot)= f  (\a(\cdot))$.
By Lemma \ref{stereoviat} (i), with $n$ replaced by $k-1$, we continue:
\be\label {lasz2}
(\frS f)(\a\gam')=(2r)^{k-1}\intl_{\tilde \bbr^{k-1} } f_\a(u'+r\rho \nu (y)) \,\frac{dy}{(|y|^2 +1)^{k-1}}.\ee
The argument of $f_\a$ can be transformed as follows:
\bea
&&u'+r\rho \nu (y)=u' + \left[ \begin{array} {cc} I_{n-1} &  0  \\ 0 & \rho_\psi
\end{array}\right] \frac{2y+(|y|^2-1)\,e_{n+1}}
{|y|^2+1}\, \sin \psi\nonumber\\
&&=(e_n \sin \psi \! +\!e_{n+1}\, \cos \psi) \cos \psi \nonumber\\
&&+ \frac{\sin \psi}{|y|^2+1} [2y \!-\!e_n (|y|^2\!-\!1)\,\cos \psi   \!  + \! e_{n+1}(|y|^2\!-\!1) \sin \psi] \!= \!\frac{A}{|y|^2+1},\nonumber\eea
where simple calculation yields
\[
A=2y \sin \psi+e_n \sin 2 \psi + e_{n+1} (|y|^2 + \cos 2\psi).\]
Hence,
\bea
&&(\frS f) (\gam)=(2\sin \psi)^{k-1} \nonumber\\
&&\times\intl_{\tilde \bbr^{k-1} }  \!\! \!f_\a
\left (\frac{2y \sin \psi+e_n \sin 2 \psi + e_{n+1} (|y|^2 \!+\! \cos 2\psi)}
{|y|^2+1}\right )\,\frac{dy}{(|y|^2 \!+\!1)^{k-1}}. \nonumber\eea
Changing variable $y=z  \sin \psi$ and setting $t=\cot \psi$, we  write this expression  as
\bea
&&2^{k-1}\!\!\intl_{\tilde \bbr^{k-1} } \!\!\! f_\a\left (\frac{2(z\!+\!te_n)\!+\!(|z\!+\!te_n|^2 \!-\!1) \, e_{n+1}}{|z+te_n|^2 +1}\right)\, \frac{dz}{(|z\!+\!te_n|^2 \!+\!1)^{k-1}}\nonumber\\
&& =\intl_{\tilde \bbr^{k-1} } g(\a(z+te_n)) \, dz, \quad g(x)=\frac{2^{k-1} (f\circ \nu)(x)}{(|x|^2+1)^{k-1}}. \nonumber\eea
This is exactly the Radon-John transform
\[
(R_{k-1}g)(\a \tilde \bbr^{k-1} +t\a e_n)=(R_{k-1}g)(\z_0 +v)=(R_{k-1}g)(\z), \]
where  (cf. (\ref{set1}))
\be\label {nsd} \z_0 =\t_0 \cap\rn, \qquad v=t\th=\frac{|u|\th}{\sqrt{1-|u|^2}}, \qquad  \z=\t \cap \rn.\ee

\end{proof}

\section{Existence Conditions,  Zonal Functions, Support Theorems, and Inversion }

Theorem \ref{th1} makes it possible to study the spherical slice transform $\frS $ using known properties of the Radon-John transform $R_{k-1}$. 
Below we consider several corollaries of this theorem.

\subsection{Existence Conditions}

 We recall 
Theorem 3.2 from \cite{Ru13b}, which characterizes convergence of the Radon-John transform (see also Theorem 4.24 in \cite {Ru15} for $k=n$). In our notation, this theorem reads as follows.
\begin{theorem} \label{7zvrt}  If
\be\label{for10zk} \intl_{|x|>a} \frac{|g(x)|}{|x|^{n-k+1}}\, dx
<\infty, \quad \forall \, a>0,\ee
 then $(R_{k-1}g)(\tau)$ is finite for almost all $\t\!\in \!\T$.
If  $g$ is nonnegative, radial,  and (\ref{for10zk}) fails, then $(R_{k-1}g)(\tau)\!=\!\infty$ for every $\t$ in $\T$.
\end{theorem}

The next statement extends  Theorem 7.8 from \cite[]{Ru15} from $k=n$ to all $2\le k \le n$.

\begin{proposition} \label{cor1} Let $\eta \in S^n$, $\t \in \T$, $2\le k \le n$. 
If
\be \label {stereogr1r50} \intl_{\eta_{n+1}>1-\e}
\frac{|f(\eta)|}{(1-\eta_{n+1})^{(n+1-k)/2}}\,dS(\eta)<\infty, \quad \forall \,0<\e\le 2, \ee
 then $(\frS f)(\t \cap S^n)$ is finite for almost all $\t \in \T$.
If $f$ is nonnegative, zonal, and (\ref{stereogr1r50}) fails, then  $(\frS f)(\t \cap S^n)= \infty$ for all $\t \in \T$.
 If $f$ is continuous on $S^{n}\setminus \{N\}$,  then  $(\frS f)(\t \cap S^n)$ is finite for each $\t \in \T$ provided that 
\be\label  {ConveQ}
\sup\limits_{\eta \in S^n}(1-\eta_{n+1})^{\mu} |f(\eta)| <\infty\quad \text{for some}\quad  \mu < \frac{k-1}{2},\ee
where the bound $(k-1)/2$ is sharp.
\end{proposition}
\begin{proof} By (\ref{Slik1}), the  integral in (\ref{for10zk}) can be written as
\[
 \intl_{|x|>a} \left (\frac{2}{|x|^2+1}\right )^{k-1} \frac{|(f\circ \nu)(x)|}{|x|^{n-k+1}}\,  dx.\]
 Convergence of the latter is equivalent to convergence of the integral
\[
 I=\intl_{|x|>a} \frac{|(f\circ \nu)(x)|\, dx}{(|x|^2+1)^{(n+k-1)/2}}.\]
Hence, by Lemma \ref{stereoviat} and (\ref{ll3567g}),  (\ref{for10zk}) is equivalent to (\ref{stereogr1r50}).
  The result for continuous functions easily follows if we write $I$ as
  \[I=c \intl_{\eta_{n+1}>1-\e} f(\eta) (1-\eta_{n+1})^{(k-n-1)/2} d\eta\]
  and switch to spherical coordinates; cf. \cite [formula (1.12.10)] {Ru15}.  
\end{proof}

\subsection{Zonal Functions}

The next Proposition is a generalization of Theorem 7.11 from  \cite {Ru15} for $k=n$.
\begin{proposition} \label{cor2}  Let $2\le k \le n$, and let $f$ be a zonal function satisfying (\ref{stereogr1r50}). We write $\eta \in S^n$  in spherical coordinates as
\be\label {pqw} \eta =\om\, \sin \vp + e_{n+1}\, \cos \, \vp, \qquad \om \in S^{n-1}, \qquad 0<\vp\le \pi,\ee
and set $ f(\eta)=f_0 (\cot \vp/2)$; (cf. Figure 1). Then
\[  (\frS f)(\t \cap S^n)=F_0(t), \qquad t=\frac{|\t|}{\sqrt {1-|\t|^2}}, \footnote{ We recall (see Notation) that $|\t|$ stands for the smallest 
Euclidean distance of any point of the affine plane $\t$ to the origin $(0, \ldots, 0)$.}    \]
where
\be F_0(t)=2^{k-1} \sig_{k-2}\intl^\infty_t \frac{f_0 (s)}{(1+s^2)^{k-1}}  \, (s^2 -t ^2)^{(k-3)/2}\, s\, ds, \ee
$\sig_{k-2}$ being the area of the $(k-2)$-dimensional unit sphere.
\end{proposition}
\begin{proof}
The statement follows from the similar fact for the Radon-John transform, according to equalities
\[
(\frS f)(\t \cap S^n)=(R_{k-1} g)(\t \cap \rn), \qquad g(x)=\frac{2^{k-1} (f\circ \nu)(x)}{(|x|^2+1)^{k-1}};\]
see (\ref{Slik1}). Because $f$ is zonal on $S^n$,  $f\circ \nu$ and  $g$  are radial functions on $\rn$. We set
$|x|=s$ and recall  that $|x| =\cot \vp/2$; see (\ref{ll3567g}), (\ref {pqw}). Then
\[
 (f\circ \nu)(x)= f \left ( \frac{2s e_n + (s^2\!-\!1)e_{n+1}}{s^2 +1} \right )= f_0 (s), \quad  g(x)=  \frac{2^{k-1} f_0 (s)}{(s^2+1)^{k-1}}. \]
If 
\[\t \cap \rn =\z_0 +v, \qquad \z_0 \in G_{n, k-1}, \qquad v\in \z_0^\perp \cap \rn,\]
 then, by \cite [Formula (2.6)]{Ru04b},
\[
(R_{k-1} g)(\t \cap \rn)\equiv (R_{k-1} g)(\z_0 +v)= \sig_{k-2}\intl^\infty_t \tilde g (s) \, (s^2 -t^2)^{(k-3)/2}\, s\, ds, \]
where
\[t=|v|=\frac{|\t|}{\sqrt {1-|\t|^2}}, \qquad   \tilde g (s)=\frac{2^{k-1} f_0 (s)}{(s^2+1)^{k-1}}. \]
cf. (\ref{nsd}).  This gives the result.
\end{proof}

\subsection {Support Theorems}

Our next aim is to establish connection between supports of $f$ and $\frS f$. We set  $b\in (-1,1)$,  $b_*=\sqrt {(1+b)/2}$,
 \[\Om_b=\{\eta \in S^n: \, \eta_{n+1} >b\}, \qquad    \tilde \Om_b=\{\t \in \T: |\t| > b_*\}.\]

\begin{theorem} \label{cor3}
  Let $2\le k\le n$ and suppose that $f$ satisfies (\ref{stereogr1r50}).
\vskip 0.2 truecm

\noindent $\rm (i)$  If  $f(\eta)=0$ for almost all $\eta \in \Om_b$, then $(\frS f)(\t \cap S^n) =0$ for almost all $\t \in \tilde \Om_b$.

\vskip 0.2 truecm

\noindent $\rm (ii)$ Conversely, suppose additionally that $f$ is continuous on $S^{n}\setminus \{N\}$ and
\be\label  {Conve}
\sup\limits_{\eta \in S^n}(1-\eta_{n+1})^{k-1-m/2} |f(\eta)| <\infty\quad \text{for all}\quad  m>0.\ee
Then the following implication holds:
\be\label{kzUYu1}
(\frS f)(\t \cap S^n) =0 \; \forall \,\t \in \tilde \Om_b \; \Longrightarrow \;f(\eta)=0 \; \forall \, \eta \in \Om_b.\ee
If $k<n$, then (\ref{kzUYu1}) holds under the less restrictive assumption   (\ref{ConveQ}), which is sharp.
%\be\label  {ConveQ}
%\sup\limits_{\eta \in S^n}(1-\eta_{n+1})^{\mu} |f(\eta)| <\infty\quad \text{for some}\quad  \mu < \frac{k-1}{2}.\ee
\end{theorem}
\begin{proof} 
The result follows from the similar one for the Radon-John transform in accordance with (\ref{Slik1}).
It is known \cite[p. 33]{H11} that if $g\in C(\rn)$ and $|x|^m g(x)$ is bounded on $\rn$ for all $m>0$, then, for any fixed  $a>0$, the 
equality $(R_{k-1} g)(\z)=0 \; \forall \,|\z|>a$ implies $g(x)=0\; \forall \,|x|>a$. By (\ref{Slik1})  and (\ref{ll3567g}),
\[|x|^m  g(x)=\frac{2^{k-1}\, |x|^m \, (f\circ \nu)(x)}{(|x|^2+1)^{k-1}}= f(\eta) (1\!-\!\eta_{n+1})^{k-1-m/2} (1\!+\!\eta_{n+1})^{m/2}.\]
 Suppose $|x|\equiv |\nu (\eta)|>a>0$, which is equivalent
to 
\[\frac {2}{1-\eta_{n+1}} > a^2 +1 \;\,  \text{\rm or} \;\, \eta_{n+1} >b>-1, \; \text{\rm where} \; \, b=\frac{a^2 -1}{a^2 +1}<1.\]
 Hence, if  we set $ b_* =\sqrt {(1+b)/2}$, $ \;\tilde \Om_b=\{\t \in \T: |\t| > b_*\}$,
 and assume $ (\frS f)(\t \cap S^n) =0\; \forall \,\t \in \tilde \Om_b$,  we get  $(R_{k-1} g)(\z)=0 \; \forall \,|\z|>a$. The latter implies $g(x)=0\; \forall \,|x|>a$ or  $f(\eta)=0\; \forall \,\eta \in \Om_b$.
  The statement (i) can be proved similarly.   
  If $k$ is strictly less than $n$ then, by Theorem  \ref{hobfu}, the number $m$ in the above reasoning can be replaced by  $\lam > k-1$. Setting $\mu= k-1-\lam/2$, we obtain the second statement in (ii).
\end{proof}

\begin{remark} In the case $k=n$,  the  condition  (\ref{Conve}) 
 cannot be dropped; cf. \cite[Remark 2.9 on p. 15]{H11}, \cite [Remark 4.119]{Ru15}. However, the continuity requirement can be substituted as follows.

\begin{theorem} \label {pro} {\rm (cf. \cite [Theorem 7.13] {Ru15})} Let $k=n$. If
\[ \intl_{\Om_b}  (1-\eta_{n+1})^{-1-m/2} \,|f (\eta)|\, dS(\eta)<\infty  \quad   \text{\rm for all}\; \, m\in \{1,2, \ldots \}\]
 and $(\frS f)(\t \cap S^n) =0$ for almost all $\t \in \tilde \Om_b$, then $f(\eta)=0$ for almost all $\eta \in \Om_b$.
\end{theorem}
The proof of this fact relies on the theory of spherical harmonics and the structure of the null space of the hyperplane Radon transform. 
The description of the null space of  the $k$-plane  transform in $\rn$ with $k<n-1$ is more complicated (see \cite{ER} for discussion), and the lower-dimensional analogue of Theorem \ref {pro} is unknown.
\end{remark}

 \subsection {Inversion Formulas}
Theorem \ref{th1} yields explicit inversion formulas for $\frS f$. Specifically, let us write (\ref{Slik1}) as
\be\label {iny} \frS f=AR_{k-1}Bf,\ee
where the operators $A$ and $B$ are defined by
\[
(Ah)(\t \cap S^n)= (h\circ \nu)(\t \cap S^n), \quad (Bf)(x)=\frac{2^{k-1} (f\circ \nu)(x)}{(|x|^2+1)^{k-1}}.\]
Hence, if $\frS f=F$, then
\be\label {inv} f= B^{-1}R_{k-1}^{-1} A^{-1}F,\ee
where $A^{-1}$ and $B^{-1}$ act by the rule
\bea
 (A^{-1} F)(\t \cap \rn)&=& (F \circ \nu^{-1})(\t \cap \rn), \nonumber\\
   (B^{-1} g)(\eta)&=&(1-\eta_{n+1})^{1-k} (g\circ \nu^{-1})(\eta),\nonumber\eea
\[\nu^{-1}(\eta)\equiv \nu^{-1}(\om \,\sin \vp + e_{n+1} \cos \, \vp)= \om \cot (\vp/2);\]
 see Figure 1.

  A plenty of inversion formulas  for $R_{k-1}$ are available in the literature; see, e.g., \cite{GGG1}, \cite [p. 33] {H11}, \cite{Ru04b, Ru13b, Ru15}, and references  therein.
  %We leave to the interested reader to combine them with $A^{-1}$ and $B^{-1}$.
 The choice of the analytic expression for  $R_{k-1}^{-1}$ in (\ref{inv})
 depends on the class of functions $f$ or, equivalently, on the class of functions $g=Bf$. In particular, many inversion formulas for  $R_{k-1} g$ are known 
  when $g\in L^p (\rn)$;  see an example in Appendix. By Lemma \ref{stereoviat} and (\ref{ll3567g}), the condition $g\in L^p (\rn)$ is equivalent to
\be\label{gYYvtCov}
(1-\eta_{n+1})^{k-1-n/p} f(\eta) \in L^p (S^n), \qquad 1\le p< \frac{n}{k-1}.\ee
 The restriction  $p< n/(k-1)$ is sharp because it is  inherited from the relevant sharp restriction  for $R_{k-1}g$.

 It is also known that if $k<n$, then the inversion problem for the Radon-John transform $(R_{k-1}g)(\z)$ is overdetermined if $\z$ varies over the set 
 of {\it all} $(k-1)$-planes in $\rn$. Gel'fand's celebrated question  \cite {Gelf60} asks   how to eliminate this overdeterminedness. In our case, it 
 means that we want to define an $n$-dimensional submanifold  $M$ of the set of all  $(k-1)$-planes $\z$ in $\bbr^{n+1}$ 
 so that  $g$  could be reconstructed from   $(R_{k-1}g)(\z)$ when the latter is known only for $\z$ belonging to $M$. 
 
 There are several ways to do this. For instance, one can choose $M$  to be
 the subset of all $(k-1)$-planes which are parallel to a fixed $k$-dimensional coordinate plane; see \cite {Ru15a} and references therein for details. 
 Thus, to reconstruct $f$ from $(\frS f)(\t \cap S^n)$, $\t \in \T$, it suffices to restrict $\t$ to the subset $\nu (M) \subset \T$.

In conclusion we observe that an idea of factorization of a given operator in terms of auxiliary operators with known properties (cf.  (\ref{iny})), is widely used in 
Analysis; see, e.g., \cite [Section 3.1]{P1}, where it is applied to a wide class of Funk-type transforms.

\section{Appendix}

Below we recall some known facts for the $k$-plane transforms; see, e.g., \cite {H11, Ru04b}. For the sake of convenience, we replace $k-1$ in (\ref{Slik})  by $k$ and set
\[\vp (\z)\equiv (R_{k} g)(\z)=\intl_{\z} g(x)\, d_\z x, \qquad \z\in G(n, k),\]
where $G(n, k)$ is the Grassmannian bundle  of all $k$-dimensional affine planes in $\rn$ and  $d_\z x $ stands for the Euclidean volume element in $\z$.
The {\it dual
$k$-plane transform} $R_k^*$  averages  functions $\vp$ on $G(n,k)$  over all $k$-planes  passing through  $x\in \rn$:
\be  \label{mmsdcrt} (R_k^* \vp)(x)=\intl_{O(n)}\vp(\g\zeta_0 +x) \,d\g.\ee
Here $\zeta_0$ is an arbitrary fixed $k$-plane through the origin and $d\g$ denotes the  Haar probability measure on the orthogonal group $O(n)$.

To formulate the inversion result for $R_{k} g$  in the  $L^p$ setting, we  invoke Riesz fractional derivatives, which  can be defined as  hypersingular integrals of the form
\bea
\label{invs1}(\bbd^k
h)(x)&\equiv&\frac{1}{d_{n,\ell}(k)}\intl_{\bbr^n}
\frac{(\Del^\ell_y h)(x)}{|y|^{n+ k}}\, dy\\
\label{invs1a}&=&\lim\limits_{\e \to
0}\frac{1}{d_{n,\ell}(k)}\intl_{|y|>\e} \frac{(\Del^\ell_y
h)(x)}{|y|^{n+ k}}\, dy. \eea
In our case, $h=R_k^*\vp$, 
\[(\Del^\ell_y h)(x)=\sum_{j=0}^\ell (-1)^j {\ell \choose j}
h(x-jy)\] is the finite difference of $h$ of order $\ell$,
\bea
d_{n,\ell}(k)&=&\frac{\pi^{n/2}}{2^k \Gamma ((n+k)/2)} \nonumber\\
{} \nonumber\\
&\times& \left\{ \!
 \begin{array} {ll} \! \Gamma (-k/2) B_l (k) \!  & \mbox{if $ k \neq 2, 4, 6, \ldots $,}\\
{}\\
\displaystyle{\frac{2 (-1)^{k/2-1}}{(k/2)!} \left [\frac {d}{d\a} B_l (\a)\right ]_{\a=k}}, & \mbox{if $ k = 2, 4, 6, \ldots \, ,$}\\
\end{array}
\right. \nonumber\eea
\[
B_l (\a) = \sum_{j=0}^l (-1)^{j} {l \choose j} j^{ \a}.
\]
The integer $\ell$ is arbitrary  with the choice $\ell = k$
if $k = 1, 3, 5, \ \ldots \ $ and any $ \ell > 2[k/2]$ (the integer part of $k/2$), otherwise.
 Further details can be found in 
\cite[Section 3.5]{Ru15} and \cite [Chapter 3]{Sam}.

\begin{theorem}\label{howa1fu} Let    $1 \le k \le n-1$. If $g\in L^p (\rn)$, $1\le p <n/k$, and $R_kg=\vp$, then,
\be\label{kryafeu}
g= c_{k,n}^{-1} \bbd^k R_k^*\vp, \qquad  c_{k,n}=\frac{2^{k}\pi^{k/2} \Gam (n/2)}{\Gam ((n-k)/2)},\ee
where the Riesz fractional derivative $\bbd^k$ is defined by (\ref{invs1})-(\ref{invs1a}).
The limit in (\ref{invs1a}) can be interpreted 
  in the $L^p$-norm and in the almost
 everywhere sense. If, moreover,  $g$ is continuous, then the convergence in
 (\ref{invs1a}) is uniform on $\bbr^n$.
\end{theorem}

The following support theorem is   a slight, but important, improvement of the remarkable discovery by Kurusa \cite [Theorem 3.1] {Ku94} who established 
 essential difference between the cases $k=n-1$ and $k<n-1$.

\begin{theorem}\label{hobfu} 
Let $(R_{k} g)(\z)$ be the $k$-plane transform of a continuous function $g$ on $\rn$, satisfying 
\be\label{kzafeu}
|x|^{\lam }|g(x)| \le c \quad \text{for some} \quad\lam >k, \qquad c=\const\footnote{In \cite {Ku94} it was assumed $\lam=k+1$.}.\ee
We assume $1 \le k<n-1$ and denote by $|\z|$ the Euclidean distance from $\z$ to the origin. Then for any $r>0$ the following implication holds:
\be\label{kzafeu1}
(R_{k} g)(\z)=0 \; \forall \, |\z| >r  \; \Longrightarrow \; g(x)=0 \; \forall \, |x| >r.\ee
\end{theorem}
\begin{proof} Fix any $(k+1)$-plane $T$ at a distance $t>r$ from the origin and let $\rho_T \in SO(n)$ be a rotation satisfying 
\[
\rho_T : \, \bbr^{k+1} +te_n \rightarrow T, \qquad \bbr^{k+1} = \bbr e_1 \oplus \cdots \oplus \,\bbr e_{k+1}.\]
Suppose that $\z\subset T$ and denote 
 $\z_1=\rho_T^{-1} \z \subset \bbr^{k+1} +te_n$, $\; g_T (x)=g (\rho_T x)$. 
By the assumption, $(R_{k} g)(\z)=0$ for all $\z\subset T$, and therefore 
$(R_{k} g)(\rho_T \z_1) =(R_{k} g_T)(\z_1)=0$ for all $k$-planes $\z_1$ in $ \bbr^{k+1} +te_n$. Every such plane has the form $\z_1=\z_2 +te_n$, where $\z_2$ is a $k$-plane in $\bbr^{k+1}$. Hence,
\[
0=(R_{k} g_T)(\z_2 +te_n)=\intl_{\z_2} g_T (y +te_n)\, dy= (R_{k} \tilde g)(\z_2),\]
which is the $k$-plane  transform transform of $\tilde g (y)= g_T (y +te_n)$, $y\in \bbr^{k+1}$. By (\ref{kzafeu}), $|\tilde g (y)|\le c\, |y|^{-\lam }$. Indeed,
\[ 
|y|^{\lam }|\tilde g (y)|=|y|^{\lam}\, |g_T (y +te_n)|\le \frac{c\, |y|^{\lam}}{|y+te_n|^{\lam}} \le c.\]
By the injectivity of the $k$-plane (see \cite[Theorem 3.4]{Ru13b}), the above inequality implies  $\tilde g (y)= g_T (y +te_n)=0$ for all $y\in \bbr^{k+1}$. 
Because $T$ and $t>r$ are arbitrary, the result follows.
\end{proof}

\vskip 0.2 truecm
{\bf Acknowledgement.} I am grateful to Prof. \'Arp\'ad Kurusa for discussion of Theorem \ref{hobfu}, and to Ms. Emily Ribando-Gros  for figures  in this article. 
%My special thanks go 
%to the referees who  thoughtful read the manuscript and made valuables suggestions.

\bibliographystyle{amsplain}

\end{document}